\documentclass[12pt,a4paper]{article}

\usepackage[english,francais]{babel}
\usepackage[latin1]{inputenc}
\usepackage{latexsym,amssymb,mathtools}
\usepackage{amsmath,amssymb,amsfonts,amscd,theorem}
\usepackage{enumerate,relsize,stmaryrd}
\usepackage[english]{babel}

\usepackage{hyperref}

\renewcommand{\hom}{\operatorname{Hom}}

\newcommand{\smallcup}{\operatorname{\smallsmile}}

\newcommand{\image}{\operatorname{Im}}

\newcommand{\Z}{\mathbf{Z}}
\newcommand{\C}{\mathbf{C}}
\newcommand{\Ad}{\operatorname{Ad}}

\newcommand{\tr}{\operatorname{tr}}
\newcommand{\SLn}[1][n]{\mathrm{SL}(#1)}
\newcommand{\sln}[1][n]{\mathfrak{sl}(#1)}
\newcommand{\GLn}[1][n]{\mathrm{GL}(#1)}
\newcommand{\gln}[1][n]{\mathfrak{gl}(#1)}
\newcommand{\metrep}{\varrho}

\def\co{\colon\thinspace}
\newenvironment{proof}[1][Proof]%
{
\begin{trivlist} \item[]  {\em #1.} }%
{\hspace*{\fill} $\Box$
\end{trivlist}}

\theoremstyle{change}

\newtheorem{thm}{Theorem}[section]
\newtheorem{prop}[thm]{Proposition}
\newtheorem{lemma}[thm]{Lemma}

\newtheorem{cor}[thm]{Corollary}

\theorembodyfont{\rmfamily}
       
       \newtheorem{remark}[thm]{Remark}

\title{Irreducible representations of knot groups into $\SLn[n,\C]$}

\author{Leila Ben Abdelghani and Michael Heusener}
\date{}

\begin{document}
\selectlanguage{english}

\maketitle
\begin{abstract}
The aim of this article is to study the existence of certain reducible, metabelian representations of knot groups into 
$\mathrm{SL}(n,\mathbf{C})$ which generalise the representations studied previously  by G.~Burde and G.~de Rham.
Under specific hypotheses we prove  the existence of irreducible deformations of such representations of knot groups into 
$\mathrm{SL}(n,\mathbf{C})$.

\end{abstract}

\section{Introduction }
\label{Introduction}

In  \cite{BenAbdelghani-Heusener-Jebali2010}, the authors studied the deformations of certain metabelian, reducible representations of knot groups into $\SLn[3,\C]$. In this paper we continue this study by generalizing all of 
the  results of \cite{BenAbdelghani-Heusener-Jebali2010} to the group $\SLn[n,\C]$ (see Theorem~\ref{thm:mainthm}). 

Let $\Gamma$ be a finitely generated group.
The set $R_n(\Gamma):=R(\Gamma,\SLn[n,\C])$ of homomorphisms of $\Gamma$ in $\SLn[n,\C]$ is called the $\SLn[n,\C]$-representation variety of $\Gamma$. It is a (not necessarily irreducible) algebraic variety. 
A representation $\rho\co\Gamma\to \SLn[n,\C]$ is called \emph{abelian} (resp.\ \emph{metabelian}) if the restriction of $\rho$ to the first (resp.\ second) commutator subgroup of $\Gamma$ is trivial.
The representation $\rho\co\Gamma\to \SLn$ is called \emph{reducible} if there exists a proper subspace $V\subset\C^n$ such that $\rho(\Gamma)$ preserves $V$. Otherwise $\rho$ is called \emph{irreducible}.

Let $\Gamma$ denote the 
\emph{knot group}  of the knot $K\subset S^3$ i.e.\ $\Gamma$ is the fundamental group of the knot complement of $K$ in $S^3$.
Since the ring of complex  Laurent polynomials $\C[t^{\pm1}]$ is a principal ideal domain, 
the complex \emph{Alexander module} $M(t)$ of $K$  
decomposes into a direct sum of cyclic modules. A generator of the order ideal of $M(t)$
is called the \emph{Alexander polynomial} of $K$. It will be denoted by $\Delta_K(t)\in\C[t^{\pm1}]$,
and it is unique up to multiplication by a unit $c\, t^k\in\C[t^{\pm1}]$, $c\in\C^*$, $k\in\Z$.
For a given root $\alpha\in\C^*$ of 
$\Delta_K(t)$ we let $\tau_{\alpha}$ denote the $(t-\alpha)$-torsion of the Alexander module.
(For details see Section~\ref{sec:notations}.)

The main result of this article is the following theorem which
generalizes the results of \cite{BenAbdelghani-Heusener-Jebali2010} where the case $n=3$ was investigated. It also applies in  the case $n=2$ which was studied in
\cite{BenAbdelghani2000} and~\cite[Theorem~1.1]{Heusener-Porti-Suarez2001}. 
\begin{thm}\label{thm:mainthm} Let $K$ be a knot in the $3$-sphere $S^3$.
If the $(t-\alpha)$-torsion $\tau_\alpha$ of the Alexander module is cyclic of the form $\C[t^{\pm1}]\big/(t-\alpha)^{n-1}$, $n\geq 2$, 
then for each $\lambda \in \C^*$ such that $\lambda^n=\alpha$ there exists a certain reducible metabelian representation 
$\metrep_\lambda$ of the knot group $\Gamma$ into $\SLn[n,\C]$.
Moreover, the representation $\metrep_\lambda$ is a smooth point of the representation variety $R_{n}(\Gamma)$, 
it is contained in a unique $(n^2+n-2)$-dimensional component $R_{\metrep_\lambda}$ of $R_n(\Gamma)$.
Moreover, $R_{\metrep_\lambda}$ contains irreducible non-metabelian representations which deform $\metrep_\lambda$. 
\end{thm}

\medskip

This paper is organised as follows. In Section~\ref{sec:notations} we introduce some notations and recall some facts which will be used in this article. In Section~\ref{sec:reducible} we study the existence of certain reducible representations.
These representations were previously studied in \cite{Jebali2008}, and we treat the existence results from a more general point of view.
Section~\ref{sec:CohomologicalComputations} is devoted to the proof of Proposition~\ref{prop:dimensions},
and it contains all necessary cohomological calculations. 
In the last section we prove that there are irreducible non-metabelian deformations of the initial reducible representation.

\paragraph{Acknowledgements.} Both authors are pleased to acknowledge the support by the French-Tunisian 
CMCU project \no 12G/1502. Moreover,
the first author was supported by the project Erasmus mundus E-GOV-TN. 
The second author acknowledges support from the ANR projects SGT and ModGroup.

\section{Notations and facts}
\label{sec:notations}

To shorten notation we will simply write $\SLn$ (respectively $\GLn$) instead of $\SLn[n,\C]$ (respectively $\GLn[n,\C]$) and $\sln$ 
(respectively $\gln$) instead of $\sln[n,\C]$ (respectively $\gln[n,\C]$).

\paragraph{Group cohomology.}
\label{Group cohomology and representation variety}
 The general reference for group cohomology is K.~Brown's book \cite{Brown1982}.
Let $A$ be a $\Gamma$-module. We denote by $C^*(\Gamma;A)$ the cochain complex, the coboundary 
operator $\delta\co C^n(\Gamma;A)\to C^{n+1}(\Gamma;A)$ is given by:
\begin{multline*}\delta f(\gamma_1,\ldots,\gamma_{n+1})=\gamma_1\cdot f(\gamma_2,\ldots,\gamma_{n+1})\\+\sum_{i=1}^{n}(-1)^i f(\gamma_1,\ldots,\gamma_{i-1},\gamma_i\gamma_{i+1},\ldots,\gamma_{n+1})+(-1)^{n+1} f(\gamma_1,\ldots,\gamma_n)\,.\end{multline*}
 The coboundaries (respectively cocycles, cohomology) of $\Gamma$ with coefficients in $A$ are denoted by $B^*(\Gamma;A)$ (respectively $Z^*(\Gamma;A),\ H^*(\Gamma;A)$).  
In what follows
$1$-cocycles and $1$-coboundaries  will be also called \emph{derivations} and \emph{principal derivations} respectively. 

Let $A_1,\ A_2$ and $A_3$ be $\Gamma$-modules. The cup product of two cochains $u\in C^p(\Gamma;A_1)$ and $v\in C^q(\Gamma;A_2)$ is the cochain $u\smallcup v\in 
C^{p+q}(\Gamma; A_1\otimes A_2)$ defined by
\begin{equation}\label{eq:cup}
u\smallsmile v(\gamma_1,\ldots,\gamma_{p+q}):=u(\gamma_1,\ldots,\gamma_p)\otimes\gamma_1\ldots\gamma_p\circ v(\gamma_{p+1},\ldots,\gamma_{p+q})\,.
\end{equation}
Here $A_1\otimes A_2$ is a $\Gamma$-module via the diagonal action. It is possible to combine the cup product with any $\Gamma$-invariant bilinear map $A_1\otimes A_2\to A_3$. We are mainly interested in the product map $\mathbf{C}\otimes\mathbf{C}\to\mathbf{C}$.
\begin{remark}
Notice that  our definition of the cup product \eqref{eq:cup} 
differs from the convention used in \cite[V.3]{Brown1982} by the sign $(-1)^{pq}$.
Hence with the definition \eqref{eq:cup} the following formula holds:
\[
\delta( u\smallcup v ) = (-1)^q \, \delta u \smallcup v + u \smallcup \delta v\,.
\]
\end{remark}

A short exact sequence
\[ 0\to A_1 \stackrel{i}{\longrightarrow} A_2 \stackrel{p}{\longrightarrow} A_3 \to 0\]
of $\Gamma$-modules gives rise to a short exact sequence of cochain complexes:
\[ 0 \to C^*(\Gamma;A_1)\stackrel{i^*}{\longrightarrow} C^*(\Gamma;A_2) \stackrel{p^*}{\longrightarrow} C^*(\Gamma;A_3) \to 0\,.\]
We will make use of the corresponding long exact cohomology sequence
(see \cite[III. Prop.~6.1]{Brown1982}):
\[ 0\to H^0(\Gamma;A_1)\longrightarrow H^0(\Gamma;A_2) \longrightarrow H^0(\Gamma;A_3)\stackrel{\beta^0}{\longrightarrow}
H^1(\Gamma;A_1)\longrightarrow \cdots\]
Recall that the Bockstein homomorphism 
$\beta^{n}\co H^n(\Gamma;A_3)\to H^{n+1}(\Gamma;A_1)$ is determined by the snake lemma: 
if $z\in Z^n(\Gamma;A_3)$ is a cocycle and if $\tilde z \in (p^*)^{-1}(z)\subset C^n(\Gamma;A_2)$ is any lift of $z$ then  $\delta_2 (\tilde z) \in \image(i^*)$ where $\delta_2$ the coboundary operator of $C^*(\Gamma;A_2)$. It follows that any cochain $z'\in C^{n+1}(\Gamma;A_3)$ such that 
$i^*(z') = \delta_2 (\tilde z)$ is a cocycle and that its cohomology class does only depend on the cohomology class represented by $z$. The cocycle $z'$ represents 
 the image of the cohomology class represented by $z$ under $\beta^{n}$. 
 \begin{remark}\label{rem:bockstein}
 By abuse of notation and if no confusion can arise,  we will write sometimes 
 $\beta^n(z)$ for a cocycle $z\in Z^n(\Gamma;A_3)$ even if the map $\beta^n$ is only well defined on cohomology classes. This will simplify the notations. 
 \end{remark}

\paragraph{The Alexander module}
Given a knot $K\subset S^3$, we let $X=\overline{S^3\backslash V(K)}$ denote its complement where $V(K)$ is a tubular neighborhood of $K$. Let $\Gamma=\pi_1(X)$ denote the fundamental group of $X$ and $h\co\Gamma\to\Z$, $h(\gamma)=\mathrm{lk}(\gamma,K)$, the canonical projection.
Recall also that a knot complement $X$ is aspherical (see \cite[3.F]{BZH2013}). In what follows we will identify the cohomology of the knot complement and of the knot group $\Gamma$. 

Note that there is a short exact splitting sequence
\[ 1\to\Gamma'\to\Gamma\to \langle t\mid -\rangle\to 1\]
where $\Gamma'=[\Gamma,\Gamma]$ denote the commutator subgroup of $\Gamma$ and where the surjection is given
by $\gamma\mapsto t^{h(\gamma)}$. Hence $\Gamma$ is isomorphic to the semi-direct product 
$\Gamma' \rtimes \Z$. Note that $\Gamma'$ is the fundamental group of the infinite cyclic covering 
$X_\infty$ of $X$. 
The abelian group $\Gamma'/\Gamma''\cong H_1(X_\infty,\Z)$ turns into a 
$\Z[t^{\pm1}]$-module via the action of the group of covering transformations which is isomorphic to
$\langle t\mid - \rangle$.
The $\Z[t^{\pm1}]$-module $H_1(X_\infty,\Z)$
is a finitely generated torsion module  called the \emph{Alexander module} of $K$. 
Note that there are isomorphisms of $\Z[t^{\pm1}]$-modules
\[
H_*(\Gamma; \Z[t^{\pm1}]) \cong H_*(X; \Z[t^{\pm1}]) \cong
H_*(X_\infty, \Z)
\]
where $\Gamma$ acts on $\Z[t^{\pm1}]$ via 
$\gamma\, p(t) = t^{h(\gamma)}\,p(t)$ for all $\gamma \in \Gamma$ and $p(t)\in\Z[t^{\pm1}]$.
(See \cite[Chapter~5]{Davis-Kirk2001} for more details.) 
In what follows we are mainly interested in the complex version 
$\C\otimes\Gamma'/\Gamma''\cong H_1(\Gamma; \C[t^{\pm1}])$ of the Alexander module. 
As $\C[t^{\pm1}]$ is a principal ideal domain, the Alexander module $H_1(\Gamma; \C[t^{\pm1}])$ 
decomposes into a direct sum of cyclic modules of the form
$\C[t^{\pm1}]/(t-\alpha)^k$, $\alpha\in\C^*\setminus\{1\}$ i.e.\ there exist 
$\alpha_1,\ldots \alpha_s\in\C^*$ such that
\[
H_1(\Gamma; \C[t^{\pm1}])\cong \tau_{\alpha_1}\oplus \cdots\oplus \tau_{\alpha_s}
\text{ where }
\tau_{\alpha_j} = \bigoplus_{i_j=1}^{n_{\alpha_j}} \C[t^{\pm1}] \big/(t-\alpha_j)^{r_{i_j}}
\]
denotes the $(t-\alpha_j)$-torsion of  $H_1(\Gamma;\C[t^{\pm1}])$.
A generator of the order ideal of $H_1(X_\infty,\C)$ is called the \emph{Alexander polynomial} 
$\Delta_{K}(t) \in \C [t^{\pm 1}]$ of $K$ i.e.\ $\Delta_{K}(t) $ is the product
\[
\Delta_{K}(t)=\prod_{j=1}^s \prod_{i_j=1}^{n_{\alpha_j}} (t-\alpha_j)^{r_{j_i}}\,.
\]
Notice that the Alexander polynomial is  symmetric and is well defined up to multiplication by a unit 
$c\,t^k$ of 
$\C[t^{\pm1}]$, $c\in\C^*$, $k\in\Z$. Moreover,  $\Delta_K(1)=\pm1\neq 0$
(see \cite{BZH2013}), and hence the $(t-1)$-torsion of the Alexander module is trivial.

For completeness we will state the next lemma which shows that
the cohomology groups $H^*(\Gamma; \C[t^{\pm1}]/(t-\alpha)^k)$ are determined by the Alexander module
$H_1(\Gamma; \C[t^{\pm1}])$. Recall that the action of $\Gamma$ on 
$\C[t^{\pm1}]/(t-\alpha)^k$ is induced by $\gamma\, p(t) = t^{h(\gamma)} p(t)$. 
\begin{lemma}\label{lem:H1GammaC} Let $K\subset S^3$ be a knot and $\Gamma$ its group.
Let $\alpha\in\C^*$  and let 
$\tau_{\alpha} = \bigoplus_{i=1}^{n_\alpha} \C[t^{\pm1}] \big/ (t-\alpha)^{r_i}$ denote
the $(t-\alpha)$-torsion of the Alexander module $H_1(\Gamma;\C[t^{\pm1}])$.
Then if $\alpha=1$ we have that $\tau_1$ is trivial and
\[
H^q(\Gamma; \C[t^{\pm1}]/(t-1)^k)\cong 
\begin{cases}
\C & \text{for $q=0,1$} \\
0   & \text{for $q\geq 2$.}
\end{cases}
\]
Moreover, for $\alpha\neq1$ we have:
\[
H^q(\Gamma; \C[t^{\pm1}]/(t-\alpha)^k)\cong 
\begin{cases}
0 & \text{for $q=0$ and $q\geq3$, } \\
\bigoplus_{i=1}^{n_\alpha}\C[t^{\pm1}] \big/ (t-\alpha)^{\min(k,r_i)}    & \text{for $q=1,2$.}
\end{cases}
\]
In particular,
$H^1(\Gamma; \C[t^{\pm1}]/(t-\alpha)^k)\neq 0 $ if and only 
$H_1(\Gamma;\C[t^{\pm1}])$ has non-trivial $(t-\alpha)$-torsion i.e\ if $\Delta_K(\alpha)=0$. 
\end{lemma}
\begin{proof} 
Let $M$ be a  $\C[t^{\pm1}]$-module, then by the extension of scalars  \cite[III.3]{Brown1982} we have an isomorphism 
\[H^q(\Gamma;M)\cong H^q(\hom_{\C[t^{\pm1}]}(C_*(X_\infty,\C),M).\]
Since $\C[t^{\pm1}]$ is a principal ideal domain, we can apply the universal coefficient theorem and obtain
\[ H^q(\Gamma;M)\cong \mathrm{Ext}_{\C[t^{\pm1}]}^1(H_{q-1}(X_\infty,\C),M)\oplus \hom_{\C[t^{\pm1}]}(H_q(X_\infty,\C),M).\]
Now $H_0(X_\infty, \C) \cong \C \cong \C[t^{\pm1}]/ (t-1)$ and
$H_k(X_\infty, \C)=0$ for $k\geq 2$ (see \cite[Prop.~8.16]{BZH2013}) so we can apply the above isomorphisms to the modules
$\C[t^{\pm1}]/(t-\alpha)^k$ with $\alpha=1$ or $\alpha\not=1$.
Notice also that for $\lambda\neq\alpha$ the multiplication by $(t-\lambda)$ induces 
an isomorphism of $\C[t^{\pm1}]/(t-\alpha)^k$. 
\end{proof}

\paragraph{Representation variety.}
Let $\Gamma$ be a finitely generated group. The set of all homomorphisms of $\Gamma$ into $\SLn$ has the structure of an affine algebraic set (see \cite{Lubotzky-Magid1985} for details).
In what follows this affine algebraic set will be denoted by $R(\Gamma,\SLn)$ or simply by $R_n(\Gamma)$. 
Let $\rho\co\Gamma\to \SLn$ be a representation. 
The Lie algebra $\sln$ turns into a $\Gamma$-module via 
$\Ad\circ\rho$. This module will be simply denoted by 
$\sln_\rho$. A \emph{$1$-cocycle} or \emph{derivation} $d\in Z^1(\Gamma;\sln_\rho)$ is a map 
$d\co\Gamma\to \sln$ satisfying 
\[d(\gamma_1\gamma_2)=d(\gamma_1)+\Ad\circ\rho(\gamma_1)(d(\gamma_2))\quad,\ \forall\ \gamma_1,\ \gamma_2\in\Gamma\,.\]

It was observed by Andr\'e Weil \cite{Weil1964} that there is a natural inclusion of the Zariski tangent space 
$T_\rho^{Zar}(R_n(\Gamma))\hookrightarrow Z^1(\Gamma;\sln_\rho)$. Informally speaking, given a smooth curve $\rho_\epsilon$ of representations through $\rho_0=\rho$ one gets a $1$-cocycle $d\co\Gamma\to \sln$ by defining
\[ d(\gamma) := \left.\frac{d \, \rho_{\epsilon}(\gamma)}
{d\,\epsilon}\right|_{\epsilon=0} \rho(\gamma)^{-1},
\quad\forall\gamma\in\Gamma\,.\]

It is easy to see that the tangent space to the orbit by conjugation corresponds to the space of $1$-coboundaries $B^1(\Gamma;\sln_\rho)$. Here, $b\co\Gamma\to \sln$ is a coboundary if there exists 
$x\in \sln$ such that $b(\gamma)=\Ad\circ\rho(\gamma)(x)-x$. A detailed account can be found in \cite{Lubotzky-Magid1985}.

For the convenience of the reader, we state the following result which is implicitly contained in 
\cite{BenAbdelghani-Heusener-Jebali2010,Heusener-Porti-Suarez2001,Heusener-Porti2005}. 
A detailed proof of the following streamlined version can be found in \cite{Heusener-Medjerab2014}:
\begin{prop}\label{prop:smoothpoint}
Let $M$ be an orientable, irreducible $3$-manifold with infinite fundamental group $\pi_1(M)$ and incompressible tours boundary, and let
$\rho\co\pi_1(M)\to \SLn$ be a representation.

If 
$\dim H^1(M; \sln_{\rho}) =n-1$ then $\rho$   is a smooth point of the 
$\SLn$-representation variety $R_n(\pi_1(M))$. More precisely,
$\rho$ is contained in a unique component of dimension
$n^2+n-2 - \dim H^0(\pi_1(M);\sln_\rho)$.
\end{prop}

\section{Reducible metabelian representations}
\label{sec:reducible}
Recall that every nonzero complex number $\alpha\in\C^*$ determines an action of a knot group $\Gamma$ on the complex numbers given by $\gamma\, x = \alpha^{h(\gamma)} x$ for $\gamma\in\Gamma$ and 
$x\in\C$. The resulting $\Gamma$-module will be denoted by $\C_\alpha$. Notice that
$\C_\alpha$ is isomorphic to $\C[t^{\pm1}]/(t-\alpha)$.

It is easy to see that a map
 $\Gamma\to\GLn[2,\C]$ given by
\begin{equation}\label{eq:rep-gl2}   
\gamma\mapsto
\begin{pmatrix} 1 & z_1(\gamma)\\0 & 1\end{pmatrix}
\begin{pmatrix} \alpha^{h(\gamma)} & 0\\0 & 1\end{pmatrix} =
\begin{pmatrix} \alpha^{h(\gamma)} & z_1(\gamma)\\0 & 1\end{pmatrix}
\end{equation}
is a representation 
if and only if the map $z_1\co\Gamma\to\C_\alpha$ is a derivation i.e. 
\[
\delta z_1(\gamma_1,\gamma_2)=  \alpha^{h(\gamma_1)} z_1(\gamma_2)- z_1(\gamma_1\gamma_2) + z_1(\gamma_1)=0 \text{ for all $\gamma_1,\gamma_2\in\Gamma$.}\]

The representation given by \eqref{eq:rep-gl2} is non-abelian if and only if $\alpha\not=1$ and the cocycle $z$ is not a coboundary.
Hence it follows from Lemma~\ref{lem:H1GammaC} that such a reducible non abelian representation exists if and only if $\alpha$ is a root of the Alexander polynomial.
These representations were first studied independently by G.~Burde \cite{Burde1967} and 
G.~de Rham \cite{deRham1967}.

We extend these considerations to a map
 $\Gamma\to\GLn[3,\C]$. It follows easily that
\begin{equation}\label{eq:rep-gl3}
\gamma\mapsto
\begin{pmatrix} \alpha^{h(\gamma)} & z_1(\gamma)& z_2(\gamma)\\0 & 1 & h(\gamma)
\\ 0 & 0 & 1\end{pmatrix}
\end{equation}
is a representation if and only if $\delta z_1=0$ and $\delta z_2+z_1\smallcup h =0$ i.e.
\[
\begin{cases}
\delta z_1(\gamma_1,\gamma_2)= 0 &  \text{ for all $\gamma_1,\gamma_2\in\Gamma$,} \\
\delta z_2(\gamma_1,\gamma_2) + z_1(\gamma_1) h(\gamma_2) =0 &
\text{ for all $\gamma_1,\gamma_2\in\Gamma$.} 
\end{cases}
\]
It was proved in \cite[Theorem~3.2]{BenAbdelghani-Lines2002} 
that the $2$-cocycle $z_1\smallcup h$ represents a non-trivial cohomology class in $H^2(\Gamma;\C_\alpha)$ provided that $z_1$ is not a coboundary and that
the $(t-\alpha)$-torsion of the Alexander module is semi-simple i.e.  
$\tau_\alpha = \C[t^{\pm1}]/ (t-\alpha)\oplus\cdots\oplus \C[t^{\pm1}]/ (t-\alpha)$.
Hence if we suppose that $z_1$ is not a coboundary then it is clear that a non-abelian representation  $\Gamma\to\GLn[3,\C]$ 
given by \eqref{eq:rep-gl3} can only exist if the $(t-\alpha)$-torsion $\tau_\alpha$ of the Alexander module has a direct summand of the form  $\C[t^{\pm1}]/ (t-\alpha)^s$, $s\geq 2$.

Representations $\Gamma\to\GLn[n,\C]$ of this type were studied in \cite{Jebali2008} where it was shown that the whole structure of the $(t-\alpha)$-torsion of the Alexander module can be recovered. 
Note that every metabelian representation of $\Gamma$ factors through 
the metabelian group $\Gamma'/\Gamma''\rtimes \Z$.

Let $\alpha\in\C^*$ be a non-zero complex number and $n\in\Z$, $n>1$. In what follows we consider the cyclic
$\C[t^{\pm1}]$-module $\C[t^{\pm1}] / (t-\alpha)^{n-1}$ and the semi-direct product
\[\C[t^{\pm1}] \big/ (t-\alpha)^{n-1} \rtimes\Z\]
where the multiplication is given by $(p_1,n_1)(p_2,n_2) = (p_1 +t^{n_1} p_2,n_1+n_2)$.
Let $I_{n}\in\SLn$ and $N_{n}\in\GLn$ denote  the identity matrix and  
the upper triangular Jordan normal form of a nilpotent matrix of degree $n$ respectively.
 For later use we note the following lemma which follows easily from the Jordan normal form theorem:
\begin{lemma}\label{lem:iso}
Let $\alpha\in\C^*$ be a nonzero complex number and let $\C^n$ be the $\C[t^{\pm1}]$-module 
with the action of $t^k$  given by
\begin{equation}\label{eq:action-J}
 t^k\,\mathbf{a} = \alpha^k\, \mathbf{a}\, J_{n}^k 
  \end{equation}
where $\mathbf a \in \C^n$ and $J_{n}=I_n+N_n$.%
Then the $\C[t^{\pm1}]$-module $\C^n$ is isomorphic to $\C[t^{\pm1}] / (t-\alpha)^{n}$.
\end{lemma}

There is a direct method to construct a reducible metabelian representations of 
$\C[t^{\pm1}] / (t-\alpha)^{n-1} \rtimes\Z$ into $\GLn[n,\C]$ (see \cite[Proposition~3.13]{Boden-Friedl2008}).
A direct calculation gives that
\[ (\mathbf{a},0)\mapsto
\begin{pmatrix}
1 & \mathbf{a} \\
\mathbf0 & I_{n-1}
\end{pmatrix} ,\quad (0,1) \mapsto 
\begin{pmatrix}
\alpha & \mathbf{0}  \\
\mathbf0 & J_{n-1}^{-1} 
\end{pmatrix}
\] 
defines a faithful representation  $\C[t^{\pm1}] / (t-\alpha)^{n-1}\rtimes\Z\to\GLn[n,\C]$.

Therefore, we obtain  a reducible, metabelian, non-abelian representation 
$\tilde\metrep\co\Gamma\to\GLn[n,\C]$ if
the Alexander module $H_1(X_\infty,\C)$ has a direct summand  of the form 
$\C[t^{\pm1}]\big/(t-\alpha)^s$ with $s\geq n-1 \geq 1$:
\begin{multline*} 
\tilde \metrep\co
\Gamma \cong \Gamma'\rtimes\Z\to
\Gamma'/\Gamma''\rtimes\Z \to  (\C\otimes\Gamma'/\Gamma'')\rtimes\Z\to   \\
\C[t^{\pm1}]\big/(t-\alpha)^s\rtimes\Z\to
\C[t^{\pm1}]\big/(t-\alpha)^{n-1}\rtimes\Z\to\GLn[n,\C]
\end{multline*}
given by
\begin{equation}\label{eq:rep-alt-gln}
 \tilde\metrep(\gamma) =
\begin{pmatrix}
1 & \tilde {\mathbf{z}}(\gamma)\\
0& I_{n-1}
\end{pmatrix}
\begin{pmatrix}
\alpha^{h(\gamma)} & 0\\
0& J^{-h(\gamma)}_{n-1}
\end{pmatrix}.
\end{equation}
It is easy to see that a map $\tilde\varrho\co\Gamma\to\GLn$ given by \eqref{eq:rep-alt-gln} is a homomorphism if and only if 
$\tilde{\mathbf{z}}\co\Gamma\to \C^{n-1}$ is a cocycle i.e.\ for all $\gamma_1,\gamma_2\in\Gamma$ we have
\begin{equation} \label{eq:tildez}
\tilde{\mathbf{z}}(\gamma_1\gamma_2) = \tilde{\mathbf{z}}(\gamma_1)
+\alpha^{h(\gamma_1)} \tilde{\mathbf{z}}(\gamma_2) J_{n-1}^{h(\gamma_1)}\,.
\end{equation}

For a better description of the cocycle $\tilde{\mathbf{z}}$, we introduce the following notations:
for $m,\ k\in\Z$, $k\geq 0$, we define
\begin{equation}\label{eq:def-h_k}
h_k(\gamma) :=  {h(\gamma)\choose k}
\quad\text{ where }\quad 
 {m\choose k} :=
 \begin{cases}
 \frac{m(m-1)\cdots(m-k+1)}{ k! } & \text{ if $k>0$}\\
 1 & \text{ if $k=0$.}
 \end{cases}
\end{equation}
It follows directly from the properties of the binomial coefficients that for each $k\in\Z$, $k\geq0$, the cochains $h_k\in C^1(\Gamma;\C)$ are defined and verify:
\begin{equation}\label{eq:h_k}
 \delta h_k + \sum_{i=1}^{k-1} h_i\smallcup h_{k-i} =0.
\end{equation}

\begin{lemma} \label{lem:tildez}
Let $\tilde{\mathbf{z}}\co\Gamma\to \C^{n-1}$ be a map verifying \eqref{eq:tildez} and
let $\tilde z_k\co\Gamma\to \C_\alpha$, $\tilde {\mathbf{z}}=(\tilde z_1,\ldots,\tilde z_{n-1})$, denote
the components of $\tilde {\mathbf{z}}$. Then the cochains $\tilde z_k$, $1\leq k\leq n-1$, satisfy
\[
  \delta \tilde{z}_k + \sum_{i=1}^{k-1} h_i\smallcup \tilde{z}_{k-i} = 0\,.
\]
In particular $\tilde{z}_1\co\Gamma\to\C_\alpha$ is a cocycle.
\end{lemma}
\begin{proof}
Note that $h_0\equiv 1$, $h_1=h$, 
$J_{n-1}^m=(I_{n-1}+N_{n-1})^m = \sum_{i\geq0}  {m\choose i} N^i_{n-1}$ and
$(x_1,\ldots,x_{n-1}) J_{n-1}^m = (x'_1, x'_2 ,\ldots, x'_{n-1})$  where 
\[
x'_k = \sum_{i=0}^{k-1} {m\choose i} x_{k-i} = x_k + \sum_{i=1}^{k-1} {m\choose i} x_{k-i}\,.
\]

It follows from this formula that $\tilde{\mathbf{z}}(\gamma_1\gamma_2) = \tilde{\mathbf{z}}(\gamma_1)
+\alpha^{h(\gamma_1)} \tilde{\mathbf{z}}(\gamma_2) J_{n-1}^{h(\gamma_1)}$ holds
if and only if for $k=1,\ldots,n-1$ we have
\[  
\tilde{z}_k (\gamma_1\gamma_2) = 
\tilde{z}_k(\gamma_1)+ \alpha^{h(\gamma_1)} \tilde{z}_{k}(\gamma_2)+
\sum_{i=1}^{k-1} h_i(\gamma_1)\,\alpha^{h(\gamma_1)}\tilde{z}_{k-i}(\gamma_2) \,.
\]
In other words $0=\delta \tilde{z}_{k} + \sum_{i=1}^{k-1} h_i\smallcup\tilde{z}_{k-i}$ holds.
\end{proof}

From now on we will suppose that for $\alpha\in\C^*\setminus\{1\}$ the $(t-\alpha)$-torsion of the Alexander module is cyclic of the form \[\tau_\alpha=\C[t^{\pm1}]\big/(t-\alpha)^{n-1}\,,\qquad\text{where}\ n\geq 2\,.\]
This is equivalent to the fact that $\alpha$ is a root of the Alexander polynomial $\Delta_K(t)$ 
of multiplicity $n-1$ and that $\dim H^1(\Gamma;\C_\alpha)=1$ 
(see Lemma~\ref{lem:H1GammaC}). Let us recall also that by Lemma~\ref{lem:H1GammaC}, 
the following dimension formulas hold:
\begin{align}\label{eq:dimHGammaC}
\dim H^q(\Gamma;\C) =
\begin{cases}
1 & \text{ for $q=0,1$;}\\
0 & \text{ for $q\geq2$,}
\end{cases}
\intertext{and}
\label{eq:dimHGammaCa}
\dim H^q(\Gamma;\C_{\alpha^{\pm1}}) =
\begin{cases}
1 & \text{ for $q=1,2$;}\\
0 & \text{ for $q\neq1,2$.}
\end{cases}
\end{align}

\begin{remark}\label{rem:blanchfield}
Notice that by Blanchfield-duality the 
$(t-\alpha^{-1})$-torsion of the Alexander module $H_1(\Gamma; \C[t^{\pm1}])$ is also of the form
$$\tau_{\alpha^{-1}} = \C[t^{\pm1}]/(t-\alpha^{-1})^{n-1}.$$ 
More precisely, 
the Alexander polynomial $\Delta_K(t)$ is symmetric and hence $\alpha^{-1}$ is also a root of 
 $\Delta_K(t)$ of multiplicity $n-1$ and  $\dim H^1(\Gamma;\C_{\alpha^{-1}})=1$.
\end{remark}

Let $ \tilde\metrep\co\Gamma\to\GLn$ be a representation given by \eqref{eq:rep-alt-gln} i.e.\
for all $\gamma\in\Gamma$ we have
\[
 \tilde\metrep(\gamma) =
\begin{pmatrix}
1 & \tilde {\mathbf{z}}(\gamma)\\
0& I_{n-1}
\end{pmatrix}
\begin{pmatrix}
\alpha^{h(\gamma)} & 0\\
0& J^{-h(\gamma)}_{n-1}
\end{pmatrix}.
\]
We will say that $\tilde\metrep$ can be \emph{upgraded} to a representation into $\GLn[n+1,\C]$ if there is a cochain $\tilde z_n\co\Gamma\to\C_\alpha$ such that the map $\Gamma\to\GLn[n+1,\C]$
given by
\[
\gamma \mapsto
\begin{pmatrix}
1 & (\tilde{\mathbf{z}}(\gamma),\tilde z_n(\gamma))\\
0& I_{n}
\end{pmatrix}
\begin{pmatrix}
\alpha^{h(\gamma)} & 0\\
0& J^{-h(\gamma)}_{n}
\end{pmatrix}
\]
is a representation.

\begin{lemma}\label{lem:no-promotion}
Suppose that the $(t-\alpha)$-torsion of the Alexander module is cyclic of the form
$\tau_\alpha=\C[t^{\pm1}]\big/(t-\alpha)^{n-1}$, $n\geq2$ and let 
$\tilde \metrep\co\Gamma\to\GLn[n,\C]$ be a representation given by \eqref{eq:rep-alt-gln}.

Then $\tilde\metrep$ cannot be upgraded to a representation into $\GLn[n+1,\C]$   unless
$\tilde z_1\co\Gamma\to\C_\alpha$ is a coboundary.
\end{lemma}
\begin{proof}
By Lemma~\ref{lem:iso} the $\C[t^{\pm1}]$-module $\C^{n-1}$ with the action given by 
$t\,\mathbf{a} = \alpha\, \mathbf{a}\, J_{n-1}$ is isomorphic to $\C[t^{\pm1}]/(t-\alpha)^{n-1}$. 
Hence it follows from the universal coefficient theorem that for $l \geq n-1$ we have:
\begin{align*}
H^1(\Gamma; \C[t^{\pm1}]/(t-\alpha)^{l}) & \cong
\mathrm{Hom}_{\C[t^{\pm1}]} \big( H_1(\Gamma; \C[t^{\pm1}]), \C[t^{\pm1}]/(t-\alpha)^{l} \big) \\ &\cong
\mathrm{Hom}_{\C[t^{\pm1}]} \big(\C[t^{\pm1}]/(t-\alpha)^{n-1}, \C[t^{\pm1}]/(t-\alpha)^{l} \big) \\ &\cong
(t-\alpha)^{l-n+1}  \C[t^{\pm1}]/(t-\alpha)^{l} \cong  \C[t^{\pm1}]/(t-\alpha)^{n-1}\,.
\end{align*}
Hence if $l>n-1$ then every cocycle $\tilde z\co\Gamma\to\C[t^{\pm1}]/(t-\alpha)^{l}$, given by
$\tilde z(\gamma)=(\tilde z_{1}(\gamma),\ldots,\tilde z_{l}(\gamma))$ 
is cohomologous to a cocycle for which the first $l-n+1$ components vanish. 
This proves the conclusion of the lemma.
\end{proof}

Notice that the unipotent matrices 
$J_{n}$ and $J_{n}^{-1}$ are similar: a direct calculation shows that $P_n J_n P_n^{-1} =J_n^{-1} $ where $P_n=(p_{ij})$, $p_{ij}=(-1)^j {j\choose i}$.
The matrix $P_n$ is upper triangular with $\pm1$ in the diagonal and $P_n^2$ is the identity matrix, and therefore $P_n=P_n^{-1}$.

Hence $\tilde\metrep$ is conjugate to a representation 
$\metrep\co\Gamma\to\GLn[n,\C]$ given by
\begin{equation}\label{eq:rep-gln}
\metrep(\gamma) =
\begin{pmatrix}
\alpha^{h(\gamma)} &  z(\gamma)\\
0& J^{h(\gamma)}_{n-1}
\end{pmatrix}
=\begin{pmatrix}
\alpha^{h(\gamma)}&z_1(\gamma)&z_2(\gamma)&\ldots&z_{n-1}(\gamma)\\
0&1&h_1(\gamma)&\ldots&h_{n-2}(\gamma)\\
\vdots&\ddots&\ddots&\ddots&\vdots\\
\vdots&&\ddots&1&h_1(\gamma)\\
0&\ldots&\ldots&0&1
\end{pmatrix}
\end{equation}
where $\mathbf{z}=(z_1,\ldots,z_{n-1})\co\Gamma\to\C^{n-1}$ satisfies
\[
\mathbf{z}(\gamma_1\gamma_2) =
\alpha^{h(\gamma_1)}\mathbf{z}(\gamma_2) + \mathbf{z}(\gamma_1)J_{n-1}^{h(\gamma_2)}\,.
\]
It follows directly that $\mathbf{z}(\gamma)=\tilde{\mathbf{z}}(\gamma) P_{n-1} J_{n-1}^{h(\gamma)}$ and in particular $z_1= - \tilde z_1$.

The same argument as in the proof of Lemma~\ref{lem:tildez} shows that the cochains $z_k\co\Gamma\to\C_\alpha$  verify:
\[ \delta z_k + \sum_{i=1}^{k-1} z_i\smallcup h_{k-i} =0 \quad 
\text{for $k=1,\ldots,n-1$.}
\]

Therefore, the representation $\metrep\co\Gamma\to\GLn[n,\C]$ can be upgraded into a representation $\Gamma\to\GLn[n+1,\C]$ if and only if $\sum_{i=1}^{n-1} z_i\smallcup h_{n-i}$ is a coboundary.

Hence we obtain the following: 
\begin{prop}\label{prop:no-promotion}
Suppose that the $(t-\alpha)$-torsion of the Alexander module is cyclic of the form
$\tau_\alpha=\C[t^{\pm1}]\big/(t-\alpha)^{n-1}$, $n\geq2$.
Let $\tilde\metrep,\metrep\co\Gamma\to\GLn[n,\C]$ be the representations given by 
\eqref{eq:rep-alt-gln} and  \eqref{eq:rep-gln} respectively where 
$\tilde{z}_1= - z_1\co\Gamma\to\C_\alpha$ is a non-principal derivation.
Then the representations $\tilde\metrep$ and $\metrep$
can not be upgraded to  representations
$\Gamma\to\GLn[n+1,\C]$ i.e. the cocycles
\[
\sum_{i=1}^{n-1} h_i\smallcup\tilde{z}_{n-i}\quad\text{ and}\quad
\sum_{i=1}^{n-1} z_i\smallcup h_{n-i}
\]
represent  nontrivial cohomology classes in $H^2(\Gamma;\C_\alpha)$.
\end{prop}
\begin{proof}
The proposition follows from Lemma~\ref{lem:no-promotion} and the above considerations.
\end{proof}

\section{Cohomological computations}
\label{sec:CohomologicalComputations}
We suppose throughout this section that $K\subset S^3$ is a knot and that the $(t-\alpha)$-torsion of its Alexander module is cyclic of the form $\tau_\alpha=\C[t,t^{-1}]\big/(t-\alpha)^{n-1}$, $n\geq2$, where $\alpha\in\C^*$ is a nonzero complex number.
Let $\metrep\co\Gamma\to\GLn$ be a representation given by 
\eqref{eq:rep-gln} where $z_1\co\Gamma\to\C_\alpha$ is a non-principal derivation:
\[ \metrep(\gamma) =
\begin{pmatrix}
\alpha^{h(\gamma)} &  z(\gamma)\\
0& J^{h(\gamma)}_{n-1}
\end{pmatrix}
=\begin{pmatrix}
\alpha^{h(\gamma)}&z_1(\gamma)&z_2(\gamma)&\ldots&z_{n-1}(\gamma)\\
0&1&h_1(\gamma)&\ldots&h_{n-2}(\gamma)\\
\vdots&\ddots&\ddots&\ddots&\vdots\\
\vdots&&\ddots&1&h_1(\gamma)\\
0&\ldots&\ldots&0&1
\end{pmatrix}.
\]

We choose an $n$-th root $\lambda$ of $\alpha=\lambda^n$ and we define a reducible metabelian representation 
$\metrep_\lambda\co\Gamma\to\SLn$ by
\begin{equation}\label{eq:rep-sln}
\metrep_\lambda(\gamma) = \lambda^{- h(\gamma)} \metrep(\gamma)
\end{equation}

The aim of the following sections is to 
calculate the cohomological groups of $\Gamma$ with coefficients in the Lie algebra $\sln_{\Ad\circ\metrep_\lambda}$.
Notice that the action of $\Gamma$ via $\Ad\circ \metrep$ and $\Ad\circ \metrep_\lambda$ preserve $\sln$ and coincide 
since the center of $\GLn$ is the kernel of $\Ad\co\mathrm{GL}(n)\to\mathrm{Aut}(\gln)$. Hence we have the following isomorphisms of 
$\Gamma$-modules:
\begin{equation}\label{eqn:center}
\sln_{\Ad\circ\metrep_\lambda}\cong\sln_{\Ad\circ\metrep} \quad\text{ and }\quad
\gln_{\Ad\circ\metrep} = \sln_{\Ad\circ\metrep} \oplus \C\, I_n
\end{equation}
where $\Gamma$ acts trivially on the center ${\C}I_n$ of $\gln$.
We will prove the following result:
\begin{prop}\label{prop:dimensions}
Let $K\subset S^3$ be a knot and suppose that the $(t-\alpha)$-torsion of the Alexander module of $K$ is of the form $\displaystyle\tau_\alpha=\C[t^{\pm1}]\big/(t-\alpha)^{n-1}$. Then for the representation $\metrep_\lambda\co\Gamma\to\SLn$ we have
$H^0(\Gamma;\sln_{\Ad\circ\metrep_\lambda})=0$ and
\[
\dim H^1(\Gamma;\sln_{\Ad\circ\metrep_\lambda})=\dim H^2(\Gamma;\sln_{\Ad\circ\metrep_\lambda})=n-1\,.
\]
\end{prop}
Notice that Propositions~\ref{prop:dimensions} and \ref{prop:smoothpoint} will proof the first part of Theorem~\ref{thm:mainthm}.
The proof of Proposition~\ref{prop:dimensions} will occupy the rest of this section.

Throughout this section we will consider $\gln$ as a $\Gamma$-module via $\Ad\circ\metrep$ and for simplicity we will
write $\gln$ for $\gln_{\Ad\circ\metrep}$.
It follows form Equation~\eqref{eqn:center} that 
\[ H^*(\Gamma; \gln)\cong H^*(\Gamma; \sln)\oplus H^*(\Gamma; \C)\,.\]

In order to compute the cohomological groups $H^*(\Gamma , \gln)$ and describe the cocycles, 
we will construct and use an adequate filtration of the coefficient algebra $\gln$.

\subsection{The setup}\label{sec:setup}
Let $(E_1,\ldots,E_n)$ denote the canonical basis of the space of  column vectors. 
Hence $E_i^j := E_i \, {}^tE_j$, $1\leq i, j \leq n$, form the canonical basis of $\gln$. 

Note that for $A\in \GLn$,
$\Ad_A( E_i^j) = (AE_i) ({}^tE_j A^{-1})$. The Lie algebra $\gln$ turns into a $\Gamma$-module via 
$\Ad\circ\metrep$ i.e.\ for all $\gamma\in\Gamma$  we have
\[
\gamma\cdot E_i^j = (\metrep(\gamma) E_i ) ({}^tE_j\metrep(\gamma^{-1}) )\,.
\]
Explicitly we have 
\begin{align} 
\gamma \cdot E_1^{1} &=  \left(\begin{array}{c}\alpha^{h(\gamma)} \\0 \\ \vdots \\0\end{array}\right)
\left(\alpha^{-h(\gamma)}, z_1(\gamma^{-1}), \ldots , z_{n-1}(\gamma^{-1})\right)\notag \\
&= E_1^{1} + \alpha^{h(\gamma)}z_1(\gamma^{-1}) E_1^{2} + \cdots 
+  \alpha^{h(\gamma)}z_{n-1}(\gamma^{-1}) E_1^{n};  \label{eq:action1}\\
\intertext{for $1<k\leq n$:}
\gamma\cdot E_1^{k} &= \alpha^{h(\gamma)} E_1^{k} +\alpha^{h(\gamma)}h_1(\gamma^{-1}) E_1^{k+1} + \cdots+
\alpha^{h(\gamma)}h_{n-k}(\gamma^{-1}) E_1^{n} ;  \label{eq:action2}\\
\gamma\cdot E_k^{1} &= \left(\begin{array}{c}z_{k-1}(\gamma) \\h_{k-2}(\gamma) \\ \vdots \\ h_1(\gamma)\\1\\0\\ \vdots
\end{array}\right)
\left(\alpha^{-h(\gamma)} , z_1(\gamma^{-1}), \ldots , z_{n-1}(\gamma^{-1})\right)\label{eq:action3}\\ 
\intertext{and for  $1<i,j \leq n$:}
\gamma \cdot E_i^{j} &= 
\left(\begin{array}{c}z_{i-1}(\gamma) \\h_{i-2}(\gamma) \\ \vdots \\ h_1(\gamma)\\1\\0\\ \vdots
\end{array}\right)
\left(0,\ldots,0,1, h_1(\gamma^{-1}), \ldots , h_{n-j}(\gamma^{-1})\right).\label{eq:action4}
\end{align}

For a given family $(X_i)_{i\in I}$, $X_i\in \gln$, we let $\langle X_i | i\in I \rangle\subset \gln$ 
denote the subspace of $\gln$ generated by the family. 
\begin{remark}\label{remark:cobord}
A first consequence of these calculations is that
if $c\in C^1(\Gamma;\C)$ is a cochain, then for $2\leq i\leq n$ and $1\leq j\leq n$ we have:
\[
\delta^\mathfrak{gl}(cE_i^j)=(\delta c)E_i^j+(h_1\smallcup c) E_{i-1}^j
+\cdots+(h_{i-2}\smallcup c) E_{2}^j+(z_{i-1}\smallcup c) E_{1}^j+ x
\]
where $x\co\Gamma\times\Gamma\to \langle E_k^l \mid 1\leq k\leq i,\ j < l\leq n\rangle$ is 
a $2$-cochain.
Here $\delta^\mathfrak{gl}$ and $\delta$ denote the coboundary operators of 
$C^1(\Gamma;\mathfrak{gl}(n))$ and $C^1(\Gamma;\C)$ respectively.
\end{remark}

In what follows we will also make use of the following $\Gamma$-modules:
for $0\leq i \leq n-1$, we define $C(i) = \langle E_{k}^{l} \mid 1\leq k \leq n,\, n-i\leq l \leq n\rangle$.
We have
\begin{equation}\label{eq:C(i)}
C(i)= 
\left\{
\begin{pmatrix}
0 & \cdots & 0& c_{1,n-i} & \cdots & c_{1,n} \\
0 & \cdots & 0& c_{2,n-i} & \cdots & c_{2,n} \\
\vdots & \vdots & \vdots& \vdots& & \vdots\\
0 & \cdots & 0& c_{n-1,n-i} & \cdots & c_{n-1,n} \\
0 & \cdots & 0 & c_{n,n-i} & \cdots & c_{n,n} 
\end{pmatrix} : c_{i,j}\in\C
\right\}
\end{equation}
and $\gln=C(n-1)\supset C(n-2)\supset\dots\supset C(0)=\langle E^n_1,\ldots, E^n_n\rangle\supset C(-1)=0$.

We will denote by $X+C(i)\in C(k)/C(i)$ the class represented by $X\in C(k)$, $0\leq i<k\leq n-1$.

\subsection{Cohomology with coefficients in $C(i)$}
\label{sec:C(i)}

The aim  of this subsection is to prove that for $0\leq i\leq n-2$ the cohomology groups 
$H^\ast(\Gamma; C(i) )$ vanish (see Proposition~\eqref{prop:H^qC(i)}).
First we will prove this for $i=0$ and in order to conclude we will apply the isomorphism 
$C(0) \cong C(i)/C(i-1)$ (see Lemma~\ref{lem:C(i+1)/C(i)}). Finally Lemma~\ref{lem:cobord} permits us to compute a certain Bockstein operator.

\begin{lemma}\label{lem:H^qE_1^n} The vector space $\left\langle E_1^n \right\rangle$ is a submodule of $C(0)$
and thus of $\gln=C(n-1)$
and we have
\[H^0(\Gamma;\left\langle E_1^n \right\rangle) = 0,\ \dim H^1(\Gamma;\left\langle E_1^n \right\rangle) = \dim H^2(\Gamma;\left\langle E_1^n \right\rangle)=1.\]
More precisely, the cocycles $z_1\,E_1^n\in Z^1(\Gamma;\left\langle E_1^n \right\rangle)$ and
\[
\big(\sum_{i=1}^{n-1} z_i\smallcup h_{n-i}\big)\, E^n_1 \in Z^2(\Gamma;\left\langle E_1^n \right\rangle)
\]
represent  generators of $H^1(\Gamma;\left\langle E_1^n \right\rangle)$ and $H^2(\Gamma;\left\langle E_1^n \right\rangle)$ respectively.
\end{lemma}

\begin{proof}
The isomorphism $\left\langle E_1^n \right\rangle\cong\C_\alpha$ and
Lemma~\ref{lem:H1GammaC}  imply the dimension formulas.
The form of the generating cocycles follows from the isomorphism
$\left\langle E_1^n \right\rangle\cong \C_\alpha$ and
Proposition~\ref{prop:no-promotion}.
\end{proof}

\begin{lemma}\label{lem:H^qC(0)/<E_1^n>}
The $\Gamma$-module $C(0)/\left\langle E_1^n \right\rangle$ is isomorphic to  $\C[t^{\pm 1}]/(t-1)^{n-1}$.
In particular, we obtain:
\begin{enumerate}
\item for $q=0,1$ $\dim H^q\big(\Gamma; C(0)/\left\langle E_1^n \right\rangle\big)=1$ and
$H^2\big(\Gamma; C(0)/\left\langle E_1^n \right\rangle\big)=0$,
\item the vector $E_2^n$ represents a generator of $H^0\left(\Gamma;C(0)/\left\langle E_1^n \right\rangle\right)$ and
the cochain $\bar v_1\co \Gamma\to C(0)$ given by
\[ 
\bar v_1 (\gamma) = h_1(\gamma) E_n^{n} + h_2(\gamma) E_{n-1}^{n}+\cdots+
h_{n-2}(\gamma) E_{2}^{n}
\]
represents a generator of $H^1(\Gamma;C(0)/\left\langle E_1^n \right\rangle)$.
\end{enumerate}

\end{lemma}

\begin{proof} 
First notice that $C(0)/\left\langle E_1^n \right\rangle$ is a $(n-1)$-dimensional vector space. More precisely,
a basis of this space is represented by the elements 
\[ E^{n}_{n},E^{n}_{n-1},\ldots,E^{n}_{2}\,.\] 

It follows from \eqref{eq:action4} that the action of $\Gamma$ on $C(0)/\left\langle E_1^n \right\rangle$ factors through
$h\co\Gamma\to\Z$. More precisely, we have for all $\gamma\in\Gamma$ such that
$h(\gamma)=1$ and for all $0\leq l \leq n-1$
\[ 
\gamma\cdot E^{n}_{n-l} = E^{n}_{n-l} + E^{n}_{n-l-1} 
\]
Here we used the fact that if $h(\gamma)=1$ then $h_i(\gamma)=0$ for all $2\leq i \leq n-1$.

On the other hand
\[ \left(1=(t-1)^0,(t-1),\ldots,(t-1)^{n-2}\right) \]
represents a basis of $\C[t^{\pm 1}]/(t-1)^{n-1}$ and we have for all $\gamma\in\Gamma$ such that
$h(\gamma)=1$:
\[
\gamma\cdot (t-1)^{l} = (t-1)^{l} + (t-1)^{l+1} + p 
\]
where $p\in (t-1)^{n-1}\C[t^{\pm 1}] $ and $0\leq l \leq n-2$. Hence the bijection
\[\varphi\co \{ (t-1)^{l} \mid 0\leq l\leq n-2\}\to\{ E^{n}_{n-l} \mid 0\leq l\leq n-2\}\]
given by $\varphi \co (t-1)^{l}\mapsto E^{n}_{n-l}$,
$0\leq l\leq n-2$, induces an isomorphism of $\Gamma$-modules
\[ 
\varphi \co \C[t^{\pm 1}]/(t-1)^{n-1}\xrightarrow{\ \cong\ } C(0)/\left\langle E_1^n \right\rangle\,.
\]
Now, the first assertion follows from Lemma~\ref{lem:H1GammaC}.
 
Moreover, it follows from the above considerations that $E_2^n$ represents a generator of 
$H^0(\Gamma; C(0)/\left\langle E_1^n \right\rangle)$.
To prove the second assertion consider
 the following short exact sequence
\[ 
0\to \C[t^{\pm 1}]/(t-1)^{n-2}\xrightarrow{(t-1)\cdot} \C[t^{\pm 1}]/(t-1)^{n-1} \to \C\to0
\] 
which gives the following long exact sequence in cohomology:
\begin{multline*}\label{suitecoho:cmoinsn}
0\to H^0(\Gamma; \C[t^{\pm 1}]/(t-1)^{n-2})\xrightarrow{\cong} H^0(\Gamma;\C[t^{\pm 1}]/(t-1)^{n-1})
\to\\ H^0(\Gamma;\C)\xrightarrow{\beta^0}H^1(\Gamma;\C[t^{\pm 1}]/(t-1)^{n-2})
\to\\ H^1(\Gamma;\C[t^{\pm 1}]/(t-1)^{n-1})\xrightarrow{\cong} H^1(\Gamma;\C){\to} H^2(\Gamma; \C[t^{\pm 1}]/(t-1)^{n-2})=0\,.
\end{multline*}
The isomorphisms and the vanishing of $H^2(\Gamma; \C[t^{\pm 1}]/(t-1)^{n-2})$ follow directly from 
Lemma~\ref{lem:H1GammaC}. 

Hence the Bockstein operator $\beta^0$ is an isomorphism: the element $e_{0}=1\in\C[t^{\pm 1}]/(t-1)^{n-1}$ projects onto a generator
of $H^0(\Gamma;\C)$ and if $\delta^{n-1}$ denotes the coboundary operator of 
$C^*(\Gamma; \C[t^{\pm 1}]/(t-1)^{n-1})$ we obtain:
\begin{align*} 
\delta^{n-1}(e_{0})(\gamma) &= (\gamma-1)\cdot e_{0}\\
&= h_1(\gamma) e_1 + h_2(\gamma) e_2 + \cdots + h_{n-2}(\gamma) e_{n-1}\\
&= (t-1)\cdot \big(  h_1(\gamma) e_0 + h_2(\gamma) e_1 + \cdots + h_{n-2}(\gamma) e_{n-2}\big)\,.
\end{align*}
Hence the cocycle 
$\gamma\mapsto h_1(\gamma) e_0 + h_2(\gamma) e_1 + \cdots + h_{n-2}(\gamma) e_{n-2}$
represents a generator of $H^1(\Gamma;\C[t^{\pm 1}]/(t-1)^{n-2})$.
To conclude, recall that the isomorphism $\C[t^{\pm 1}]/(t-1)^{n-1}\cong C(0)/\left\langle E_1^n \right\rangle$
is induced by the map $\varphi\co e_l\mapsto E^{n}_{n-l}$, $0\leq l\leq n-2$.
\end{proof}

\begin{lemma}\label{lem:C(i+1)/C(i)} 
For $i\in\Z$, $0\leq i\leq n-3$, the $\Gamma$-module $C(i+1)/C(i)$ is isomorphic to 
$C(0)$.
\end{lemma}

\begin{proof}
It follows from \eqref{eq:action4} that, for all $i\in\Z$, $0\leq i\leq n-2$, the bijection
 \[\phi\co\{E_{n-j}^{n-(i+1)}+C(i)\ |\ 0\leq j\leq n-1\}\to
\{E_{n-j}^{n}\ |\ 0\leq j\leq n-1\}\] 
given by $\phi(E_{n-j}^{n-(i+1)}+C(i))=E_{n-j}^{n}$ induces an isomrphism of
$\Gamma$-modules $\phi:C(i+1)/C(i)\to C(0)$.
\end{proof}

Let us recall the definition of the cochains $h_i\in C^1(\Gamma;\C)$, given by 
$h_i(\gamma) = {h(\gamma)\choose i}$ (see Equation~\eqref{eq:def-h_k}).
Recall also that for $1\leq i\leq n-1$ the 
cochains $h_i\in C^1(\Gamma;\C)$ verify Equation~\eqref{eq:h_k}:
\[ \delta h_i + \sum_{j=1}^{i-1} h_j\smallcup h_{i-j} =0.\]
\begin{lemma}\label{lem:cobord}
Let $\delta^\mathfrak{gl}$ denote the coboundary operator of $C^*(\Gamma;\gln)$.
Then for all  $0\leq k \leq n-2$ there exists a cochain 
$x_{k-1}\in C^2(\Gamma; C(k-1) )$ such that
\[
\delta^\mathfrak{gl}\big(\sum_{i=2}^{n} h_{n-i+1} E^{n-k}_i\big) =
\big(\sum_{i=1}^{n-1} z_{i}\smallcup h_{n-i}\big) E^{n-k}_1 +x_{k-1}
\]
\end{lemma}

\begin{proof} Equation~\eqref{eq:action4} 
and Remark~\ref{remark:cobord} imply that
\begin{multline*}
 \delta^\mathfrak{gl}(h_{n-i+1} E^{n-k}_i)= \\ z_{i-1}\smallcup h_{n-i+1} \, E_1^{n-k} +
\sum_{l=2}^{i-1} h_{i-l}\smallcup h_{n-i+1} \, E^{n-k}_l + \delta h_{n-i+1}\,E^{n-k}_i + x_{i,k-1} \,
\end{multline*}
where $x_{i,k-1}\in C^2(\Gamma; C(k-1))$ and $\delta$ is the boundary operator of
$C^*(\Gamma;\C)$.
Therefore,
\begin{align*}
\delta^\mathfrak{gl}(\sum_{i=2}^{n} h_{n-i+1} E^{n-k}_i) &=
\big(\sum_{i=2}^{n} z_{i-1}\smallcup h_{n-i+1}\big) E^{n-k}_1 +
 \sum_{i=2}^{n}\sum_{l=2}^{i-1} h_{i-l}\smallcup h_{n-i+1} \, E^{n-k}_l\\
&\qquad\qquad + \sum_{i=2}^{n} \delta h_{n-i+1}\,E^{n-k}_i + x_{k-1} \,.
\end{align*}
where $x_{k-1}=\sum_{i=2}^n x_{i,k-1}\in C^2(\Gamma; C(k-1))$.
A direct calculation gives that
\begin{align*}
\sum_{i=2}^{n}\sum_{l=2}^{i-1} h_{i-l}\smallcup h_{n-i+1} \, E^{n-k}_l &=
\sum_{l=2}^{n-1}\sum_{i=l+1}^{n} h_{i-l}\smallcup h_{n-i+1} \, E^{n-k}_l \\
&= \sum_{l=2}^{n-1}\Big(\sum_{i=1}^{n-l} h_{i}\smallcup h_{n-l+1-i} \Big)\, E^{n-k}_l\,.
\end{align*}
Thus
\begin{multline*}
\delta^\mathfrak{gl}(h_{n-i+1} E^{n-k}_i) =
\big(\sum_{i=1}^{n-1} z_{i}\smallcup h_{n-i}\big) E^{n-k}_1 \\ + \delta h_1\,E^{n-k}_n + 
\sum_{i=1}^{n-2} \big( \delta h_{n-i} + \sum_{l=1}^{n-i-1} h_{l}\smallcup h_{n-i-l} \big)\,E^{n-k}_i 
+ x_{k-1}\,.
\end{multline*}
Now $\delta h_1 =0$ and by \eqref{eq:h_k} we have 
$\delta h_{n-i} + \sum_{l=1}^{n-i} h_{l}\smallcup h_{n-i+1-l}=0$. Hence we obtain the claimed formula.
\end{proof}

\begin{prop}\label{prop:H^qC(i)}
For all $i\in\Z$, $0\leq i\leq n-2$ and $q\geq0$ we have
\[H^q(\Gamma;C(i))=0.\]
\end{prop}

\begin{proof}
For $q\geq 3$ we have $H^q(\Gamma;C(i))=0$ since the knot exterior $X$ has the homotopy type of a $2$-dimensional complex.
We start by proving the result for $i=0$. Consider the short exact sequence
\begin{equation}\label{seq:C(0)}
0\to \left\langle E_1^n \right\rangle\rightarrowtail C(0)\twoheadrightarrow C(0)/\left\langle E_1^n \right\rangle\to 0.
\end{equation}
As the $\C[t^{\pm 1}]$-modules $\langle E_1^n \rangle$ and 
$\C_\alpha \cong \C[t^{\pm 1}]/(t-\alpha)$ are isomorphic,
the sequence~\eqref{seq:C(0)} gives us a long exact sequence in cohomology:
\begin{multline*}
0= H^0(\Gamma;\left\langle E_1^n \right\rangle)\to H^0(\Gamma;C(0))\to 
H^0\left(\Gamma;C(0)/\left\langle E_1^n \right\rangle\right)\xrightarrow{\beta_0^0} \\
H^1(\Gamma;\left\langle E_1^n \right\rangle)
\to H^1(\Gamma;C(0))
\to H^1\left(\Gamma;C(0)/\left\langle E_1^n \right\rangle\right)\xrightarrow{\beta_0^1} \\
H^2(\Gamma;\left\langle E_1^n \right\rangle)\to 
H^2(\Gamma;C(0))\to H^2\left(\Gamma;C(0)/\left\langle E_1^n \right\rangle\right)\to 0\,.
\end{multline*}
Here, for $q=0,1$, we denoted by $\beta_0^q: H^q\left(\Gamma;C(0)/\left\langle E_1^n \right\rangle\right)
\to H^{q+1}(\Gamma;\left\langle E_1^n \right\rangle)$
the Bockstein homomorphism.
By Lemma~\ref{lem:H^qC(0)/<E_1^n>}, $E_2^n$ represents a generator of $H^0\left(\Gamma;C(0)/\left\langle E_1^n \right\rangle\right)$, so
\begin{align*}
\beta_0^0(E_2^n ) (\gamma) &=
(\gamma-1)\cdot(E_2^n) \\
&= \gamma\cdot E_2^n-E_2^n=z_1(\gamma)E_1^n.
\end{align*}
By Lemma~\ref{lem:H^qE_1^n}
 $z_1\,E_1^n$ is a generator of $H^1\left(\Gamma;\left\langle E_1^n \right\rangle\right)$, and by Lemma~\ref{lem:H^qC(0)/<E_1^n>}
$\dim H^0\left(\Gamma;C(0)/\left\langle E_1^n \right\rangle\right)=1=\dim H^1\left(\Gamma;\left\langle E_1^n \right\rangle\right)$,
thus $\beta_0^0$ is an isomorphism. Consequently $H^0(\Gamma;C(0))=0$ as $H^0\left(\Gamma;\left\langle E_1^n \right\rangle\right)=0$ by
Lemma~\ref{lem:H^qE_1^n}.

Now by Lemma~\ref{lem:H^qC(0)/<E_1^n>}, the cochain $\bar v_1\co \Gamma\to C(0)$ given by
\[ 
\bar v_1 (\gamma) = h_1(\gamma) E_n^{n} + h_2(\gamma) E_{n-1}^{n}+\cdots+
h_{n-1}(\gamma) E_{2}^{n}
\]
represents a generator of $H^1(\Gamma;C(0)/\left\langle E_1^n \right\rangle)$ and by Lemma~\ref{lem:cobord}
\[\beta_0^1\left(h_1E_n^{n} + h_2 E_{n-1}^{n}+\cdots+
h_{n-1}E_{2}^{n}\right)=
\big(\sum_{i=1}^{n-1} z_{i}\smallcup h_{n-i}\big) E^{n}_1\,.\]
Moreover, by Proposition~\ref{prop:no-promotion} the cocycle $\big(\sum_{i=1}^{n-1} z_{i}\smallcup h_{n-i}\big) E^{n}_1$ represents a generator of $H^2(\Gamma;\left\langle E_1^n \right\rangle)$.
 Thus $\beta_0^1$ is an isomorphism and $H^q(\Gamma;C(0))=0$ for $q=1,2$.

Now suppose that $H^q(\Gamma; C(i_0))=0$ for $0\leq i_0\leq n-3$, $q=0,1,2$ and consider the following
short exact sequence of $\Gamma$-modules:
\begin{equation}\label{seq:C(i_0)}
0\to C(i_0)\rightarrowtail C(i_0+1)\twoheadrightarrow C(i_0+1)/C(i_0)\to 0\,.
\end{equation}
This sequence induces a long exact sequence in cohomology 
\begin{multline*}
0\to H^0(\Gamma;C(i_0))\to H^0(\Gamma;C(i_0+1))\to 
H^0\left(\Gamma;C(i_0+1)/C(i_0)\right)\rightarrow \\
H^1(\Gamma;C(i_0))\to H^1(\Gamma;C(i_0+1))
\to H^1\left(\Gamma;C(i_0+1)/C(i_0)\right)\rightarrow\\
H^2(\Gamma;C(i_0))\to 
H^2(\Gamma;C(i_0+1))\to H^2\left(\Gamma;C(i_0+1)/C(i_0)\right)\to 0\,.
\end{multline*}
Using the hypothesis, we conclude that the  groups $H^q(\Gamma;C(i_0+1))$ and 
$H^q\left(\Gamma;C(i_0+1)/C(i_0)\right)$ are isomorphic for
$q=0,1,2$. By Lemma~\ref{lem:C(i+1)/C(i)}, we obtain
$H^q(\Gamma;C(i_0+1))\cong H^q(\Gamma;C(0))=0$ for $q=0,1,2$.
\end{proof}

\subsection{Cohomology with coefficients in $\gln$}
\label{sec:gln}
In this subsection we will prove Proposition~\ref{prop:dimensions}.

\begin{proof}[Proof of Proposition~\ref{prop:dimensions}]
In order to compute the dimensions of the cohomology groups $H^*(\Gamma;\gln)$, 
we consider the short exact sequence
\begin{equation}\label{seq:C(n-1)}
0\to C(n-2)\rightarrowtail C(n-1)=\gln\twoheadrightarrow\gln/C(n-2)\to 0\,.
\end{equation}
The sequence \eqref{seq:C(n-1)} gives rise to the following long exact cohomology sequence: 
\begin{multline*}
0\to H^0(\Gamma;\gln)\to H^0(\Gamma;\gln/C(n-2))\to 
H^1(\Gamma;C(n-2))\rightarrow \\
H^1(\Gamma;\gln)\to H^1(\Gamma;\gln/C(n-2))
\to H^2\left(\Gamma;C(n-2)\right)\rightarrow\\
H^2(\Gamma;\gln)\to 
H^2(\Gamma;\gln/C(n-2))\to 0\,.
\end{multline*}
As $H^q(\Gamma;C(n-2))=0$ we conclude that 
\[
H^q(\Gamma;\gln)\cong H^q(\Gamma;\gln/C(n-2))\,.\] 
It remains to understand the quotient $\gln/C(n-2)$.

Clearly the vectors $ E_n^1,\ldots, E_1^1$ represent  a basis of $\gln/C(n-2)$ and 
there exists a $\Gamma$-module $M$ such that
the following sequence
\begin{equation}\label{seq:last}
0\to \left\langle E_1^1+C(n-2)\right\rangle\rightarrowtail\gln/C(n-2)\twoheadrightarrow M\to 0
\end{equation}
is exact. Now the sequence \eqref{seq:last} induces the following  exact cohomology sequence:
\begin{multline}\label{seq:longlast}
0\to H^0(\Gamma;\left\langle E_1^1+C(n-2)\right\rangle)\to H^0(\Gamma;\gln/C(n-2))\to H^0(\Gamma;M)\to\\
H^1(\Gamma;\left\langle E_1^1+C(n-2)\right\rangle)\rightarrow 
H^1(\Gamma;\gln/C(n-2))\to H^1(\Gamma;M)\to \\
H^2\left(\Gamma;\left\langle E_1^1+C(n-2)\right\rangle\right)\rightarrow
H^2(\Gamma;\gln/C(n-2))\to 
H^2(\Gamma;M)\to 0\,.
\end{multline}
Observe that the action of $\Gamma$ on $\left\langle E_1^1+C(n-2)\right\rangle$ is trivial.
Therefore, $\left\langle E_1^1+C(n-2)\right\rangle$ and $\C$ are isomorphic  $\Gamma$-modules.
By 
Lemma~\ref{lem:H1GammaC} we obtain 
\[ 
\dim H^q(\Gamma;\left\langle E_1^1+C(n-2)\right\rangle)=1\quad \text{ for $q=0,1$}
\]
 and $H^2(\Gamma;\left\langle E_1^1+C(n-2)\right\rangle)=0$.

To complete the proof we will make use of Lemma~\ref{lem:last colum}, which states that
the $\Gamma$-module $M$ is isomorphic to $\C[t^{\pm 1}]/(t-\alpha^{-1})^{n-1}$.
Recall that Lemma~\ref{lem:H1GammaC} implies that $H^0(\Gamma;\C[t^{\pm 1}]/(t-\alpha^{-1})^{n-1})=0$ and
\[ 
\dim H^q(\Gamma;\C[t^{\pm 1}]/(t-\alpha^{-1})^{n-1})=n-1, 
\quad\text{ for $q=1,2$.}
\]

Therefore, sequence~\eqref{seq:longlast} gives:
\[
H^q(\Gamma;\gln)\cong H^q(\Gamma;\gln/C(n-2))\cong
\begin{cases}
H^0(\Gamma;\C) & \text{ for $q=0$};\\
H^2(\Gamma;M) & \text{ for $q=2$} 
\end{cases}
\]
and the short exact sequence:
\[
0\to H^1(\Gamma;\C)\rightarrowtail H^1(\Gamma;\gln/C(n-2))\cong H^1(\Gamma;\gln)
\twoheadrightarrow H^1(\Gamma;M)\to 0\,.
\]
\end{proof}

\begin{lemma}\label{lem:last colum} 
The $\Gamma$-module $M$ is isomorphic to $\C[t^{\pm 1}]/(t-\alpha^{-1})^{n-1}$.
Consequently
\[H^0(\Gamma;M)=0, \quad \dim H^q(\Gamma;M)=n-1,\ q=0,1.\]
\end{lemma}
\begin{proof}[Proof of Lemma~\ref{lem:last colum}] The proof is similar to the proof of Lemma~\ref{lem:H^qC(0)/<E_1^n>}.
As a $\C$-vector space the dimension of  $M$ is $n-1$ and a basis is given by
$\left( \overline{E_{n}^1},\ldots, \overline{E_2^1}\right)$ where
$\overline{E}_i^1 = E_i^1 +C(n-2)\in M$ is the class  represented by  $E_i^1$, $2\leq i\leq n$.
In order to prove that $M$ is isomorphic to 
$\C[t^{\pm 1}]/(t-\alpha^{-1})^{n-1}$ observe that by \eqref{eq:action3}
\[ \gamma\cdot E^1_k = \alpha^{-h(\gamma)} 
\big( E^1_k + h_1(\gamma) E^1_{k-1} + \cdots + h_{k-2}(\gamma) E^1_{2}\big) +X_k\]
where $X_k\in  E_1^1 + C(n-2)$.
Therefore, the action of $\Gamma$ on $M$ factors through
$h\co\Gamma\to\Z$. More precisely, we have for all $\gamma\in\Gamma$ such that
$h(\gamma)=1$
\[ 
\gamma\cdot \overline{E}^{1}_{k} = \alpha^{-1} (\overline{E}^{1}_{k} + \overline{E}^{1}_{k-1})\,.
\]
 On the other hand $e_l=\big(\alpha(t-\alpha^{-1})\big)^l$, $0\leq l\leq n-2$,
represents a basis of $\C[t^{\pm 1}]/(t-\alpha^{-1})^{n-1}$ and we have for all $\gamma\in\Gamma$ such that
$h(\gamma)=1$: 
\[
\gamma\cdot e_l = \alpha^{-1} (e_l+e_{l+1}) + p \ \text{where $p\in (t-\alpha^{-1})^{n-1}\C[t^{\pm 1}] $. }
\]
Hence the bijection
$\psi\co \{ e_l \mid 0\leq l\leq n-2\}\to\{ \overline{E}^{1}_{k} \mid 2\leq k\leq n\}$
given by $\varphi \co e_l\mapsto \overline{E}^{1}_{n-l}$,
$0\leq l\leq n-2$, induces an isomorphism of $\Gamma$-modules
$\psi \co \C[t^{\pm 1}]/(t-\alpha^{-1})^{n-1}\xrightarrow{\cong} M$.

Finally, the dimension equations follow from Lemma~\ref{lem:H1GammaC} and 
Remark~\ref{rem:blanchfield}.
\end{proof}

We obtain immediately that under the hypotheses  of Proposition~\ref{prop:dimensions} the representation $\metrep_\lambda$ is a smooth
point of the representation variety $R_n(\Gamma)$. This proves the first part of Theorem~\ref{thm:mainthm}.
\begin{prop} \label{prob:smooth}
Let $K$ be a knot in the $3$-sphere $S^3$.
If the $(t-\alpha)$-torsion $\tau_\alpha$ of the Alexander module is cyclic of the form $\C[t,t^{-1}]\big/(t-\alpha)^{n-1}$, $n\geq 2$, 
then the representation $\metrep_\lambda$ is a smooth point of the representation variety $R_{n}(\Gamma)$; 
it is contained in a unique $(n^2+2n-2)$-dimensional component $R_{\metrep_\lambda}$ of $R_n(\Gamma)$.
\end{prop}

\begin{proof}
By Proposition~\ref{prop:smoothpoint} and Proposition~\ref{prop:dimensions}, the representation $\metrep_\lambda$ is contained in a unique
component $R_{\metrep_\lambda}$ of dimension $(n^2+n-2)$. Moreover,
\begin{align*}
\dim Z^1(\Gamma;\sln) & =\dim H^1(\Gamma;\sln)+\dim B^1(\Gamma;\sln)\\
	& = (n-1)+(n^2-1 )\\ & =n^2+n-2\,.
\end{align*} 
Hence the representation $\metrep_\lambda$ is 
a smooth point of $R_n(\Gamma)$ which is contained in an unique $(n^2+n-2)$-dimensional component $R_{\metrep_\lambda}$. 
\end{proof}

For a later use, we describe more precisely the derivations $v_k\co\Gamma\to\sln$, $1\leq k\leq n-1$, which represent a basis of $H^1(\Gamma;\sln)$.

\begin{cor}\label{cor:cocycles}
There exists cochains $z_1^-,\cdots,z_{n-1}^-\in C^1(\Gamma;\C_{\alpha^{-1}})$
such that $\delta z_k^- +\sum_{i=1}^{k-1}h_i\smallcup z_{k-i}^-=0$ for $k=1,\ldots,n-1$
and $z_1^-\co\Gamma\to\C_\alpha^{-1}$ is a non-principal derivation.

Moreover, there exist cochains $g_k\co\Gamma\to\C$ and $x_k\co\Gamma\to\C(n-2)$,
$1\leq k\leq n-1$, such that the cochains $v_k\co\Gamma\to\sln$ given by
\[
v_{k}=g_kE_1^1+ z_k^-E_2^1+\cdots+z_1^-E_{k+1}^1+x_k
\] 
are cocycles and represent a basis of $H^1(\Gamma;\sln)$.
\end{cor}
\begin{proof}
Recall that the vector space $M$ admits as a basis the family 
$\left( \overline{E}_{n}^1,\ldots, \overline{E}_2^1\right)$ and that
it is isomorphic to $\C[t^{\pm 1}]/(t-\alpha^{-1})^{n-1}$. Moreover it is easily seen that $M$ is isomorphic to the $\Gamma$-module of column vectors
$\C^{n-1}$ where the action is given by $t^ka=\alpha^{-k}J_{n-1}^ka$. Hence a cochain 
$\mathbf{z}^-\co\Gamma\to M$ with coordinates
$\mathbf{z}^-={}^t(z_{n-1}^-,\cdots,z_1^-)$ is a cocycle in $Z^1(\Gamma;M)$ if and only if
for all $\gamma_1,\gamma_2\in\Gamma$
\[\mathbf{z}^-(\gamma_1\gamma_2)=\mathbf{z}^-(\gamma_1)+
\alpha^{-h(\gamma_1)}J_{n-1}^{h(\gamma_1)}\mathbf{z}^-(\gamma_2).\]

It follows, as in the proof of Lemma~\ref{lem:tildez}, that this is equivalent to
\[z_k^-(\gamma_1\gamma_2)=z_k^-(\gamma_1)+\alpha^{-h(\gamma_1)}z_k^-(\gamma_2)+\sum_{i=1}^{k-1}h_i(\gamma_1)
\alpha^{-h(\gamma_1)}z_{k-i}^-(\gamma_2).\]
In other words, for $1\leq k\leq n-1$,
\[0=\delta z_k^-+\sum_{i=1}^{k-1}h_i\smile z_{k-i}^-.\]
By Remark~\ref{rem:blanchfield}, if $z_1^-\in Z^1(\Gamma;\C_{\alpha^{-1}})$ is a non-principal derivation, 
there exist cochains $z_k^-\co\Gamma\to\C_{\alpha^{-1}}$, $2\leq k\leq n-1$,  such that 
$$0=\delta z_k^-+\sum_{i=1}^{k-1}h_i\smile z_{k-i}^-\,.$$ 
Consequently, as $\dim H^1(\Gamma;M)=n-1$, the cochains
\[
\mathbf{z}^-_k= z_k^-\overline{E}_{2}^1+\cdots +z_1^-\overline{E}_{k+1}^1, \quad 1\leq k\leq n-1, 
\]
represent a basis of $H^1(\Gamma;M)$. The proof is completed by noticing that the projection
$H^1(\Gamma;\gln)\to H^1(\Gamma;M)$ restricts to an isomorphism between
$H^1(\Gamma;\sln)$ and $H^1(\Gamma;M)$.
\end{proof}

\section{Irreducible $\SLn$ representations}\label{sec:irred}
This section will be devoted to the proof of the last part of Theorem~\ref{thm:mainthm}.
At first, we proved that the representation $\metrep_\lambda$ is a smooth
point of $R_n(\Gamma)$ which is contained in a unique $(n^2+n-2)-$dimensional component $R_{\metrep_\lambda}$.
Then, to prove the existence
of irreducible representations in that component, we will make use of Corollary~\ref{cor:cocycles}
and Burnside's theorem on matrix algebras.

\begin{proof}[Proof of the last part of Theorem~\ref{thm:mainthm}]
To prove that the component $R_{\metrep_\lambda}$ contains irreducible 
non metabelian representations, we will generalize the argument given in \cite{BenAbdelghani-Heusener-Jebali2010} for $n=3$.

Let $\Gamma=\langle S_1,\ldots,S_n|\;W_1,\ldots,W_{n-1}\rangle$ be a Wirtinger presentation of the knot group. 
Modulo conjugation
of the representation $\metrep_\lambda$, we can assume that $z_1(S_1)=\ldots=z_{n-1}(S_1)=0$. This conjugation corresponds to 
adding a coboundary to the cochains $z_i$, $1\leq i\leq n-1$.
We will also assume that the second Wirtinger generator $S_2$ verifies $z_1(S_2)=b_1\not=0=z_1(S_1)$. This is always possible
since $z_1$ is not a coboundary. 
Hence
\[\metrep_\lambda(S_1)=\alpha^{-1/n}\left(\begin{array}{c|cc}
\alpha&0\\
\hline0&J_{n-1}
\end{array}\right)\quad\text{and}\quad\metrep_\lambda(S_2)=
\alpha^{-1/n}\left(\begin{array}{c|cc}
\alpha&b\\
\hline0&J_{n-1}
\end{array}\right)\]
where  $b=(b_1,\ldots,b_{n-1})$ with $b_1\in\C^*$ and $b_i=z_i(S_2)\in\C$ for $2\leq i\leq n-1$.

Let $v_{n-1}\in Z^1(\Gamma;\sln)$ be a cocycle such that:
\[v_{n-1}= g_{n-1} E^1_1 + z_1^-E_n^1+z_2^-E^1_{n-1}+\ldots+z_{n-1}^-E_2^1+x_{n-1}\]
given by Corollary~\ref{cor:cocycles}. Up to adding a coboundary to the cocycle $z_1^-$ we assume that
$z_1^-(S_1)=0$. Notice that, by Lemma 5.5 of \cite{BenAbdelghani-Heusener-Jebali2010}, $z_1^-(S_2)\not=0$.

Let $\rho_t$ be a deformation of $\metrep_\lambda$ with leading term $v_{n-1}$:
\[
\rho_t = \big( I_n + t\, v_{n-1} + o(t) \big) \metrep_\lambda\,, \text{ where } \lim_{t\to 0} \frac{o(t)}{t} =0\,.
\]
We may apply the following lemma (whose proof is completely analogous to that of Lemma 5.3 in~\cite{BenAbdelghani-Heusener-Jebali2010})
to this deformation for $A(t)=\rho_t(S_1)$.
\begin{lemma}\label{lem:eigen}

Let $\rho_t\co\Gamma\to \SLn$ be a curve in $R_n(\Gamma)$ with $\rho_0=\metrep_\lambda$. Then there exists a curve $C_t$ in $\SLn$ such that $C_0=I_n$ and 
\[\Ad_{C_t}\circ\rho_t(S_1)=\begin{pmatrix}
a_{11}(t)&0&\ldots&0\\
0&a_{22}(t)&\ldots&a_{2n}(t)\\
\vdots&\vdots&&\vdots\\
0&a_{n2}(t)&\ldots&a_{nn}(t)
\end{pmatrix}\,\]
for all sufficiently small $t$.
\end{lemma}

Therefore,  we may suppose that $a_{n1}(t)=0$, and since
\[ 
a_{n1}(t)=t\lambda^{n-1}\left(z_1^-(S_1)+\delta c(S_1)\right)+o(t)\,,\ \text{ for $c\in\C$,}
\]
 it follows that
\[a'_{n1}(0)=\lambda^{n-1}(z_1^-(S_1)+(\alpha^{-1}-1)c)=0\]
and hence $c=0$. For $B(t)=\rho_t(S_2)$, we obtain $b'_{n1}(0)=\lambda^{n-1}z_1^-(S_2)\neq 0$. Hence, we can apply 
the following technical lemma (whose  proof will be postponed to the end of this section).

\begin{lemma}\label{lem:irred}
Let $A(t)=(a_{ij}(t))_{1\leq i,j\leq n}$ and $B(t)=(b_{ij}(t))_{1\leq i,j\leq n}$ be matrices depending analytically on $t$ such that 
\[
A(t)=\left(\begin{array}{c|cc}
a_{11}(t)& 0\\
\hline 0& A_{11}(t)&
\end{array}\right),\quad
A(0) =\metrep_\lambda(S_1)=
\alpha^{-1/n}\left(\begin{array}{c|cc}
\alpha&0\\
\hline0&J_{n-1}
\end{array}\right)
\]
and
\[
B(0) =\metrep_\lambda(S_2)=
\alpha^{-1/n}\left(\begin{array}{c|cc}
\alpha&b\\
\hline0&J_{n-1}
\end{array}\right)\,.
\]
If the first derivative $b'_{n1}(0)\neq 0$ then for sufficiently small $t$, $t\neq 0$, the matrices $A(t)$ and $B(t)$ generate the full matrix algebra $M(n,\mathbf{C})$.
\end{lemma}

Hence for sufficiently small $t\neq0$ we obtain that $A(t)=\rho_t(S_1)$ and 
$B(t)=\rho_t(S_2)$ generate $M(n,\mathbf{C})$.
By Burnside's matrix theorem, such a representation $\rho_t$ is irreducible

To conclude the proof of Theorem~\ref{thm:mainthm}, we will prove that all irreducible representations sufficiently close to $\metrep_\lambda$ are non-metabelian.
In order to do so, we will make use of the following result of H.~Boden and S.~Friedel 
\cite[Theorem 1.2]{Boden-Friedl2008}: for 
every irreducible metabelian representation $\rho\co\Gamma\to \SLn$ we have $\tr\rho(S_1)=0$. 
Now, we have $\tr\metrep_\lambda(S_1)=\lambda^{-1}(\lambda^n+n-1)$ and we claim that
$\lambda^n+n-1\neq 0$. Notice that $\alpha=\lambda^n$ is a root of the Alexander polynomial
$\Delta_K(t)$ and $\lambda^n+n-1=0$ would imply that $1-n$ is a root of $\Delta_K(t)$.
This would imply that $t+n-1$ divides $\Delta_K(t)$ and hence $n$ divides $\Delta_K(1)=\pm1$ which is impossible since $n\geq 2$. Therefore, $\tr(\rho(S_1))\neq0$ for all irreducible representations sufficiently close to
$\metrep_\lambda$. This proves Theorem~\ref{thm:mainthm}.  
\end{proof}

\begin{remark}
Let $\rho_\lambda\co\Gamma\to \SLn$ be the diagonal representation given by 
$\rho_\lambda(\mu) =  \mathrm{diag}(\lambda^{n-1},\lambda^{-1} I_{n-1})$ where $\mu$ is a meridian of $K$. 
The orbit $\mathcal O (\rho_\lambda)$ of
$\rho_\lambda$ under the action of conjugation of $\SLn$ is contained in the closure 
$\overline{\mathcal O (\metrep_\lambda)}$. Hence $\metrep_\lambda$ and $\rho_\lambda$ project to the same point $\chi_\lambda$ of the variety of characters $X_n(\Gamma)= R_n(\Gamma)\sslash \SLn$.

It would be natural to study the local picture of the variety of characters
$X_n(\Gamma)= R_n(\Gamma)\sslash \SLn$ at $\chi_\lambda$ as done in \cite[\S\ 8]{Heusener-Porti2005}. Unfortunately, there are much more technical difficulties since in this case the quadratic cone $Q(\rho_\lambda)$ coincides with the Zariski tangent space 
$Z^1(\Gamma; \sln_{\rho_\lambda})$. Therefore the third obstruction has to be considered. 
\end{remark}

\begin{proof}[Proof of lemma~\ref{lem:irred}] 
The proof follows exactly the proof of Proposition 5.4  in \cite{BenAbdelghani-Heusener-Jebali2010}. We denote by $\mathcal{A}_t\subset \gln$ the algebra generated by $A(t)$ and $B(t)$. 
For any matrix $A$ we let $P_A(X)$ denote its characteristic polynomial.
We have $P_{A_{11}(0)} = (\lambda^{-1} - X)^{n-1}$ and $a_{11}(0)=\lambda^{n-1}$. 
Since $\alpha=\lambda^n\neq1$ we obtain $P_{A_{11}(0)}(a_{11}(0))\neq0$.
It follows that $P_{A_{11}(t)}(a_{11}(t))\neq 0$ for small $t$ and hence 
\[\displaystyle\frac{1}{P_{A_{11}(t)}(a_{11}(t))}P_{A_{11}(t)}(A(t))=\left(\begin{array}{c|cc}
1& 0\\
\hline 0& 0
\end{array}\right)=\begin{pmatrix}
1\\
0\\
\vdots\\
0\end{pmatrix}\otimes (1,0,\ldots,0)\in\C[A(t)]\subset \mathcal{A}_t\,.\]

In the next step we will prove that 
\[\mathcal{A}_t\begin{pmatrix}
1\\
0\\
\vdots\\
0\end{pmatrix}=\C^n\ \text{and}\ (1,0,\ldots,0)\mathcal{A}_t=\C^n\,,\ \text{for small\ }t\in\C^n\,.\]

It follows from this that $\mathcal{A}_t$ contains all rank one matrices since a rank one matrix can be written as $v\otimes w$ where $v$ is a column vector and $w$ is a row vector. Note also that $A(v\otimes w)=(Av)\otimes w$ and $(v\otimes w)A=v\otimes(wA)$. Since each matrix is the sum of rank one matrices the proposition follows.

Now consider the vectors 
\[(1,0,\ldots,0)A(0),\, (1,0,\ldots,0)B(0),\ldots,(1,0,\ldots,0)B(0)^{n-1}.\]
Then for $1\leq k\leq n-1:$
\[(1,0,\ldots,0)B(0)^k=\lambda^{-k}(\alpha^k,b\sum_{j=0}^{k-1}\alpha^{k-1-j}J^j)\] 
and the dimension $D$ of the vector space 
$$\langle(1,0,\ldots,0)A(0),(1,0,\ldots,0)B(0),\ldots(1,0,\ldots,0)B(0)^{n-1}\rangle$$ 
is equal to
\begin{align*}
D&=\dim\langle(\alpha,0),(\alpha,b),(\alpha^2,\alpha b+bJ),
\ldots,(\alpha^{n-1}, b\sum_{j=0}^{k-1}\alpha^{k-1-j}J^j)\rangle\\
&=\dim\langle (\alpha,0),(0,b),(0,bJ),\cdots(0,bJ^{n-2})\rangle.
\end{align*}
Here, $J=J_{n-1}=I_{n-1} +N_{n-1}$ where $N_{n-1}\in GL(n-1,\C)$ is the upper triangular Jordan normal form of a nilpotent matrix of degree $n-1$. Then a
direct calculation gives that 
\[
\dim\langle b,bJ,\ldots,bJ^{n-2}\rangle=
\dim\langle b,bN,\ldots,bN^{n-2} \rangle=
n-1\,,\ \text{as $b_1\not=0$.}
\]
Thus $\dim\langle(1,0,\ldots,0)A(0),(1,0,\ldots,0)B(0),\ldots(1,0,\ldots,0)B(0)^{n-1}\rangle=n$ and 
 the vectors 
\[(1,0,\ldots,0)A(0),\, (1,0,\ldots,0)B(0),\ldots,(1,0,\ldots,0)B(0)^{n-1}\] 
form a basis of the space of row vectors. This proves that $(1,0,\ldots,0)\mathcal{A}_t$ is the 
space of row vectors for sufficiently small $t$.

In the final step consider the $n$ column vectors 
\[a_1(t)=A(t)\begin{pmatrix}
1\\
0\\
\vdots\\
0
\end{pmatrix},\ a_i(t)=A^i(t)B(t)\begin{pmatrix}
1\\
0\\
\vdots\\
0
\end{pmatrix},\, 0\leq i\leq n-2
\]
and write $B(t)\begin{pmatrix}
1\\
0\\
\vdots\\
0
\end{pmatrix}=\left(\begin{array}{c}
b_{11}(t)\\
 \mathbf b(t)\end{array}\right)$ for the first column of $B(t)$; then
\[
a_1(t)=\left(\begin{array}{c}
a_{11}(t)\\
 \mathbf 0\end{array}\right),\,
a_{i+2}(t)=A^i(t)\left(\begin{array}{c}
b_{11}(t)\\
 \mathbf b(t)\end{array}\right),\,  0\leq i\leq n-2.
\]
Define the function $f(t):=\det(a_1(t),\ldots,a_n(t))$ and $g(t)$ by:
\[
f(t)=a_{11}(t) g(t), \text{ where $g(t) = \det\left( \mathbf b(t),A_{11}(t) \mathbf b(t),\ldots,
A^{n-2}_{11}(t) \mathbf b(t)\right).$}
\] 
Now, for $k\geq0$ the $k$-th derivative $g^{(k)}(t)$ of $g(t)$ is given by:
\[
\sum_{s_1,\ldots,s_{n-1}} c_{s_1,\ldots,s_{n-1}}
\det\left( \mathbf b^{(s_1)}(t),
(A_{11}(t) \mathbf b(t))^{(s_2)},\ldots,(A^{n-2}_{11}(t) \mathbf b(t))^{(s_{n-1})}\right)
\]
where 
\[ 
c_{s_1,\ldots,s_{n-1}} = 
\begin{cases}
{k\choose s_1,\ldots,s_{n-1}}=\frac{k!}{s_1!\ldots s_{n-1}!} &\text{ if $s_1+\cdots+s_{n-1}=k$;}\\
 0 &\text{ othewise.}
 \end{cases}
 \]
As $ \mathbf b(0)=0$ one have, for $0\leq k\leq n-2$,
$g^{(k)}(0)=0$ and consequently $f^{(k)}(0)=0$ for all
 $0\leq k\leq n-2$.\\
Now, for $k=n-1$, we have 
\begin{align*}
\frac{g^{(n-1)}(0)}{(n-1)!}&=\det\left( \mathbf b^\prime(0),
(A_{11}(t) \mathbf b(t))^\prime(0),\ldots,(A^{n-2}_{11}(t) \mathbf b(t))^\prime(0)\right)\\
&=\det\left( \mathbf b^\prime(0),
A_{11}(0) \mathbf b^\prime(0),\ldots,A^{n-2}_{11}(0) \mathbf b^\prime(0)\right)\\
&=\det\left( \mathbf b^\prime(0),
(\lambda^{-1}J) \mathbf b^\prime(0),\ldots,(\lambda^{-1}J)^{n-2} \mathbf b^\prime(0)\right)\\
&=\det\left( \mathbf b^\prime(0),
\lambda^{-1}N \mathbf b^\prime(0),\ldots,\lambda^{-(n-2)}N^{n-2} \mathbf b^\prime(0)\right)\\
&\not=0 \ \text{ since $b^\prime_{n1}\not=0$.}
\end{align*}
Thus, $f^{(n-1)}(0)=a_{11}(0)g^{(n-1)}(0)\neq 0$ and $f(t)\neq0$ for sufficiently small $t$, $t\neq 0$.
\end{proof}

\bibliographystyle{plain}
\bibliography{SLn}
\end{document}